\documentclass[10pt]{amsart}
\usepackage{latexsym}
\usepackage{amsfonts}
\usepackage[headings]{fullpage}

\usepackage{amsmath,amsthm,amssymb}
\usepackage{epsf}
\usepackage{psfrag}
\usepackage{graphicx}

\author[Ilya Kapovich]{Ilya Kapovich}

\address{\tt Department of Mathematics, University of Illinois at
 Urbana-Champaign, 1409 West Green Street, Urbana, IL 61801, USA}
 \email{\tt kapovich@math.uiuc.edu}

\title[Integral weight realization theorem]{An integral weight realization theorem for subset currents on free groups}

\newtheorem{theor}{Theorem}

\newtheorem{thm}{Theorem}[section] \newtheorem{lem}[thm]{Lemma}
\newtheorem{cor}[thm]{Corollary} 
\newtheorem{prop}[thm]{Proposition} \theoremstyle{definition}
\newtheorem{defn}[thm]{Definition}

\newtheorem{conv}[thm]{Convention} \newtheorem{rem}[thm]{Remark}

\newtheorem{propdfn}[thm]{Proposition-Definition}

\def\strutdepth{\dp\strutbox}
\def \ss{\strut\vadjust{\kern-\strutdepth \sss}}
\def \sss{\vtop to \strutdepth{
\baselineskip\strutdepth\vss\llap{$\diamondsuit\;\;$}\null}}

\def\strutdepth{\dp\strutbox}
\def \sst{\strut\vadjust{\kern-\strutdepth \ssss}}
\def \ssss{\vtop to \strutdepth{
\baselineskip\strutdepth\vss\llap{$\spadesuit\;\;$}\null}}

\def\strutdepth{\dp\strutbox}
\def \ssh{\strut\vadjust{\kern-\strutdepth \sssh}}
\def \sssh{\vtop to \strutdepth{
\baselineskip\strutdepth\vss\llap{$\heartsuit\;\;$}\null}}

\newcommand{\Z}{\mathbb Z}

\newcommand{\A}{{\cal A}}

\newcommand{\cal}{\mathcal}

\def\epsilon{\varepsilon}
\def\phi{\varphi}

\newcommand{\Curr}{\mbox{Curr}}
\newcommand{\Out}{\mbox{Out}}

\newcommand{\vol}{\mbox{vol}}

\newcommand{\FN}{F_N}   
\newcommand{\cvn}{\mbox{cv}_N}

\newcommand{\gcn}{\mathcal S\Curr(\FN)}
\newcommand{\gcnr}{\mathcal S\Curr_r(\FN)}

\newcommand{\pgcn}{\mathbb P\gcn}

\newcommand\isom{\mathrel{\text{%
   \setbox0\hbox{$\rightarrow$}%
   \rlap{\hbox to \wd0{\hss\raisebox{0.9\height}{$\sim$}\hss}}\box0
}}}

\begin{document}

\begin{abstract}

We prove that if $N\ge 2$ and $\alpha: F_N\to \pi_1(\Gamma)$ is a
marking on $F_N$, then for any integer $r\ge 2$ and any
$F_N$-invariant collection of non-negative integral  ``weights''
associated to all subtrees $K$ of $\widetilde \Gamma$ of radius $\le
r$ satisfying some natural ``switch'' conditions, there exists a
finite cyclically reduced folded $\Gamma$-graph $\Delta$ realizing
these weights as numbers of ``occurrences'' of $K$ in $\Delta$.  As an
application, we give a new, direct and explicit, proof of one of the
main results of our  paper with Nagnibeda~\cite{KN3} stating that for
any $N\ge 2$ the set $\gcnr$ of all rational subset currents is dense
in the space $\gcn$ of subset currents on $F_N$. (The proof given
in~\cite{KN3} was indirect and omitted significant details. The proof
given here is complete and, we hope, more accessible to the $Out(F_N)$
community).

We also answer one of the questions (Problem 10.11) posed in \cite{KN3}. Thus we  prove that if a nonzero $\mu\in \gcn$  has all weights with respect to some marking being integers, then $\mu$ is the sum of finitely many ``counting'' currents corresponding to nontrivial finitely generated subgroups of $F_N$.
\end{abstract}

\thanks{The author was supported by the NSF
  grant DMS-1405146}

\subjclass[2000]{Primary 20F, Secondary 57M, 37B, 37D}

\maketitle


\section{Introduction}

The main purpose of this paper is to give a proof of Theorem~\ref{thm:main} below (originally established in \cite{KN3} via an indirect argument) which is self-contained, direct, explicit, and can be relatively easily understood by the $Out(F_N)$ community.  We start with some history and motivation for the topic of subset currents on $F_N$.

\subsection{Shift-invariant measures and measures supported on periodic orbits}
Let $A$ be a finite alphabet consisting of $\ge 2$ letters.
One of the main objects studied in symbolic dynamics is the space $\mathcal M_A$ of all finite shift-invariant positive Borel measures on the two-sided shift space $A^\mathbb Z$. Here $A^\mathbb Z$ is the space of all bi-infinite word $\xi=\dots x_{-2} x_{-1} x_0 x_1x_2\dots x_n \dots$, where $x_i\in A$.   We also think of elements of $A^\mathbb Z$ as functions $\xi:\mathbb Z\to A$ where $\xi(i)$ is the $i$-th letter of $\xi$.

The space $A^\Z$ is endowed with the standard topology 
where two sequences $\xi_1,\xi_2\in A^\omega$ are ``close" if $\xi_1|_{[-n,\dots,n]}=\xi_2|_{[-n,\dots,n]}$ for a large $n\ge 1$.  With this topology $A^\Z$ is homeomorphic to the Cantor set.  The \emph{shift map} $T:\A^\Z\to A^\Z$ consists in shifting every bi-infinite word one letter to the left. Thus for an element $\xi:\Z\to A$ of $A^\Z$ we have $(T\xi)(i)=\xi(i+1)$, where $i\in \Z$. This map $T$ is easily seen to be a homeomorphism of $A^\Z$.  Now $\mathcal M_A$ consists of all $T$-invariant positive Borel measures on $A^\Z$ with $0\le \mu(A^\Z)<\infty$, that is, of all finite positive Borel measures $\mu$ on $A^\Z$ such that for every Borel subset $S\subseteq A^\Z$ we have $\mu(S)=\mu(T^{-1}S)$. The space $\mathcal M_A$, endowed with the weak-* topology, is a locally compact infinite-dimensional Hausdorff space.  For $\xi\in A^\Z$ the orbit $\mathcal O_T(\xi)=\{T^i\xi| i\in \mathbb Z, i\ge 0\}$  is finite if and only if the word $\xi$ is periodic, that is, has the form $\xi=\overset{\infty}{w}=\ldots wwww\dots $ for some nontrivial word $w$ over $A$ which is not a proper power in the set $A^\ast$ of all finite words over $A$. More precisely, if $w\in A^\ast$ is as above and $m=|w|\le 1$ is the length of $w$, the word $\xi=\overset{\infty}{w}$ is defined so that for any $n\in\Z$ and $j\in \{1,\dots, m\}$ with $n\equiv j \mod m$, the element $\xi(n)\in A$ is the $j$-th letter of $w$.  Then the orbit $\mathcal O_T(\xi)$ has cardinality $m$ and $\mathcal O_T(\xi)=\{ \overset{\infty}{z} \mid z \text{ is a cyclic permutation of } w\}$.  Such finite $T$-orbits are also called \emph{periodic} orbits for $T$. 
Denote by $Z(A)$ the set of all nontrivial words $w\in A^\ast$ which are not proper powers in $A^\ast$.
For every $w\in Z(A)$ there is an associated $T$-invariant measure $\mu_w\in \mathcal M_A$ supported on $\mathcal O_T(\overset{\infty}{w})$ and  defined as $\mu_w=\sum_{z} \delta_{\overset{\infty}{z}}$, where the summation is taken over all cyclic permutations $z$ of $w$.  Now if $w\in A^\ast$ is an arbitrary nontrivial word, there exist unique $k\ge 1$ and $w_1\in Z(A)$ such that $w=w_1^k$. We then define $\mu_w:=k \mu_{w_1}$. For convenience for the empty word $\epsilon\in A^\ast$ we put $\mu_\epsilon=0\in \mathcal M_A$. Then for every $w\in A^\ast$ and every $k\ge 1$ we have $\mu_{w^k}=k\mu_w$.

A key basic result of symbolic dynamics says that the set 
\[
\mathcal R_A:=\{c\mu_w| w\in Z(A), c\ge 0\}=\{c\mu_w| w\in A^\ast, c\ge 0\}
\]
 is a dense subset of $\mathcal M_A$, see \cite{Pa61,Ox63,KH}. A similar fact also holds for irreducible subshifts of finite type in $A^\Z$.  There are many different proofs of the fact that $\mathcal R_A$ is dense in $\mathcal M_A$ but, in combinatorial terms, most of these proofs rely (directly or indirectly) on a certain ``weight realization" theorem.
Since a generalization of this theorem for the case of subset currents on a free group is the main result of the present paper, we need to explain the classic case of the theorem here in more detail.
To every nonempty finite word $v\in A^\ast$ with $|w|=m\ge 1$ and an integer $n\in \Z$ we associate a \emph{cylinder} set $C_{v,n}\subseteq A^\Z$ consisting of all semi-infinite words $\xi\in A^\Z$ such that $\xi|_{[n,\dots, n+m-1]}=v$. The sets $C_{v,n}$ are compact and open in $A^\Z$ and the collection of all such cylinder sets is a basis for the standard topology on $A^\Z$ mentioned above.  If $\mu\in \mathcal A$ then, by shift-invariance of $\mu$,  for any $n\in \Z$ and any nontrivial $v\in A^\ast$ we have $\mu(C_{v,n})=\mu(C_{v,0})$.  Thus for $\mu\in \mathcal M_A$ and a nonempty word $v\in A^\ast$ we define the \emph{weight} $\langle v,\mu\rangle:=\mu(C_{v,0})$. For the empty word $\epsilon\in A^\ast$ we also put $\langle \epsilon,\mu\rangle:=\mu(A^\Z)$. 

Any measure $\mu\in \mathcal M_A$ is then uniquely determined by its collection of weights $(\langle v,\mu\rangle)_{\mu\in A^\ast}$. The fact that $\mu$ is finitely-additive  translates into the requirement that the weights satisfy the following ``switch" conditions: for every $v\in A^\ast$ we have
\[
\langle v,\mu\rangle=\sum_{a\in A} \langle va,\mu\rangle = \sum_{b\in A} \langle bv,\mu\rangle \tag{$\dag$}
\]
and hence for every $v\in A^\ast$ we have $\displaystyle\sum_{a\in A} \langle va,\mu\rangle = \sum_{b\in A} \langle bv,\mu\rangle$. Note that for every $\mu\in \mathcal M_A$ and every integer $k\ge 1$ we have $\mu(A^\Z)=\sum_{v\in A^k} \langle v,\mu\rangle$.  Kolmogorov's measure extension theorem implies that for any collection of nonnegative ``weights", indexed by elements of $A^\ast$, there exists a unique measure $\mu\in \mathcal M_A$ realizing these weights. That is, $\mathcal M_A$ has a bijective correspondence with the set of all families $\mathfrak t=(t_v)_{v\in A^\ast}$ of nonnegative real numbers such that for every $v\in A^\ast$ we have
\[
\sum_{a\in A} t_{va}=\sum_{b\in A} t_{bu}  \tag{$\spadesuit$}.
\]

For $w,v\in A^\ast$ the weight $\langle v,\mu_w\rangle=\mu_w(C_{v,0})$ has a useful combinatorial interpretation.  Namely, for a nontrivial $w\in A^\ast$ let $\overline w$ be the associated \emph{cyclic word}, that is, a directed labelled graph obtained by subdividing a circle into $m=|w|$ edges, with directed edges labelled by elements of $A$, so that going around this circle counter-clockwise once from some vertex on this graph results in reading precisely the word $w$. The graph $\underline w$ does not have a distinguished base-vertex, so that for any cyclic permutation $z$ of $w$ the graphs $\overline w$ and $\overline{z}$ are isomorphic as directed labelled graphs. For every nontrivial word $v\in A^\ast$ let $\langle v,\overline{w}\rangle$ be the number of \emph{occurrences} of $v$ in $\overline w$, that is, the number of vertices in $\overline w$ from which it is possible to ``read" the word $v$ in $\overline w$ by going counter-clockwise and while never leaving the circle $\overline w$ (we allow the path corresponding to reading $v$ in $\overline w$ to possibly begin and end at different vertices and also to possibly overlap itself). For example, if $A=\{a,b\}$ and $w=a^2$ then for every $k\ge 1$ we have $\langle a^k, \overline w\rangle=2$.  A key basic observation shows that for any nontrivial words $v,w\in A^\ast$ we have $\langle v,\overline{w}\rangle=\mu_w(C_v)=\langle v,\mu_w\rangle$.
While there are many ways to prove that the set $\mathcal R_A\subseteq \mathcal M_A$ is dense in $\mathcal M_A$, the most explicit proofs of this fact rely on the following ``integral weight realization theorem":

\begin{prop}\label{prop:iw}
Let $A$ be a finite alphabet consisting of at least two letters and let $m\ge 2$ be an arbitrary integer.
Let $\tau=(t_v)_{v\in A^m}$ be a family of non-negative integers $t_v\in \mathbb Z, t_v\ge 0$ such that for some $v\in A^m$ $t_v\ne 0$ and such that for every $u\in A^{m-1}$ we have
\[
\sum_{a\in A} t_{ua}= \sum_{b\in A} t_{bv}.\tag{$\diamondsuit$}
\]
Then there exists a finite collection of nontrivial words $w_1,\dots, w_p\in A^\ast$ such that for every $v\in A^m$ we have $\sum_{i=1}^p \langle v, \overline{w_i}\rangle=t_v$. 
\end{prop}
Proposition~\ref{prop:iw} straightforwardly implies that the set of all finite linear combinations $c_1\mu_{w_1}+\dots +c_k\mu_{w_k}$, where $k\ge 1$, $c_i\ge 0$ and $w_i\in \A^\ast$, is dense in $\mathcal M_A$. From here, with a bit of extra work, one can deduce that $\mathcal R_A=\{c\mu_w| w\in A^ast, c\ge 0\}$ is dense in $\mathcal M_A$ as well.
The proofs of Proposition~\ref{prop:iw} usually rely, in some form, on finding Euler circuits in some directed ``Rauzy  -- de Bruijn graph" $\Gamma_\tau$ (see \cite{R83,B85,FM})   associated to $\tau=(t_v)_{v\in A^m}$ as in the statement of the proposition (assuming, say, that $t_v>0$ for all $v\in A^m$). The existence of such an Euler circuit in $\Gamma_\tau$  requires checking that the in-degree of every vertex of $\Gamma_w$ is equal to the out-degree of this vertex, and this condition does hold because of equation $(\diamondsuit)$.  See, for example, \cite{Ka1,Ka2}, for the implementation of this approach in the context of (ordinary) geodesic currents on free groups.

This ``Euler circuit'' approach to proving Proposition~\ref{prop:iw}, as well as most other approaches to proving that $\mathcal R_A$ is dense in $\mathcal M_A$, significantly rely on the "commutative" or "linear" nature of the shift space $A^\Z$, that is, on the fact that every finite subword of an element of $A^\Z$ can be thought of as written on a segment of $\mathbb Z$.
A key element of the geometry of $\mathbb Z$ utilized in all of these
proofs uses the fact that every finite subsegment of $\mathbb Z$ has a
unique direction of extending it forward in $\mathbb Z$ and a unique
direction of extending it backwards in $\mathbb Z$. These approaches no longer work in the non-commutative  and highly branching context of subset currents on free groups, and we will see below that a different tool is needed to prove a version of the integral weight realization theorem there.

\subsection{Subset currents on free groups.}

In a paper with Tatiana Nagnibeda~\cite{KN3} we introduced and studied the notion of a \emph{subset current} on a free group $F_N$. This concept is motivated by that of a \emph{geodesic current}. Geodesic currents on $F_N$ are measures that generalize conjugacy classes of nontrivial elements of $F_N$. The space $Curr(F_N)$ of all geodesic currents on $F_N$ turns out to be highly useful in the study of the dynamics and geometry of $\Out(F_N)$ and of the Culler-Vogtmann Outer space, particularly via the use of the ``geometric intersection form'' constructed in~\cite{KL2}. See \cite{KN3} for an extended discussion and \cite{BF08,BR12,CK,Ha12,CP,KL2,KL3} for recent examples of  such applications.
Similarly, the notion of a subset current is a measure-theoretic analog of the conjugacy class of a nontrivial finitely generated subgroup of $F_N$.
For a free group $F_N$ let $\mathfrak C_N$ be the space of all closed subsets $S\subseteq \partial F_N$ such that $S$ consists of at least two elements. The space $\mathfrak C_N$ comes equipped with a natural topology (see Section~\ref{sec:background} below and \cite{KN3} for details) such that $\mathfrak C_N$ is a locally compact totally disconnected Hausdorff topological space. The action of $F_N$ on $\partial F_N$ by translations extends to a natural translation action of  $F_N$ on $\mathfrak C_N$ by homeomorphisms.  A \emph{subset current} on $F_N$ is a positive Borel measure $\mu$ on $\mathfrak C_N$ such that $\mu$ is finite on compact subsets and is $F_N$-invariant. The space $\gcn$ of all subset currents on $F_N$ comes equipped with a natural weak-* topology and a natural action of $\Out(F_N)$ by continuous $\mathbb R_{\ge 0}$-linear transformations.

Given a nontrivial finitely generated subgroup $H\le F_N$, there is a naturally associated \emph{counting} subset current $\eta_H\in \gcn$. The limit set $\Lambda(H)\subseteq \partial F_N$ is a closed $F_N$-invariant subset of $\partial F_N$ and, since $H\ne \{1\}$, we have $\Lambda(H)\in \mathfrak C_N$. Moreover, for any $g\in F_N$ $\Lambda(gHg^{-1})=g\Lambda(H)$.  If $H$ is equal to its commensurator $Comm_{F_N}(H)$, we define $\eta_H:=\sum_{H_1\in [H]} \delta_{\Lambda(H_1)}$, where $[H]$ is the conjugacy class of $H$ in $F_N$. For an arbitrary nontrivial finitely generated subgroup $H\le F_N$ it is known that $m:=[Comm_{F_N}(H):H]<\infty$ and that $Comm_{F_N}(H)$ is equal to its own commensurator in $F_N$. Then we define $\eta_{H}:=m\, \eta_{Comm_{F_N}(H)}$.  It is shown in \cite{KN3} that $\eta_H$ is indeed a subset current on $F_N$. A subset current $\mu\in \gcn$ is called \emph{rational} if $\mu=c\eta_H$ for some $c\ge 0$ and some nontrivial a finitely generated $H\le F_N$.
Denote by $\gcnr$ the set of all rational subset currents on $F_N$.

One can also equivalently describe 
$\eta_H$ in more combinatorial terms, using Stallings core graphs, see \cite{KN3} and Proposition-Definition~\ref{prop:wd} below. Such a combinatorial description exists for any ``marking" on $F_N$ (that is, an isomorphism $\alpha$ between $F_N$ and $\pi_1(\Gamma)$ where $\Gamma$ is a finite connected graph without any degree-1 vertices and with the first betti number equal to $N$). For the purposes of stressing the analogy with $A^\Z$ described above we will assume that $A=\{a_1,\dots, a_N\}$ is a free basis of $F_N$, that $\Gamma_A$ is a wedge of $N$ oriented loop-edges labelled by $a_1,\dots, a_N$ wedged at a vertex $p_0$ and that $F_N$ is identified with $\pi_1(\Gamma_A,p_0)$ in the natural way according to this labelling. Thus $X_A=\widetilde \Gamma_A$ is the Cayley graph of $F_N$ with respect to $A$ and $\partial F_N=\partial X_A$. The space $\mathfrak C_N$ is then canonically identified with the space $\mathfrak T_A$ of all infinite subtrees $Y$ of $X_A$ such that $Y$ has no degree-one vertices. This identification is given by sending a tree $Y\in \mathfrak T_A$ to the closed subset $\partial Y\subseteq \partial X_A=\partial F_N$, so that $\partial Y\in \mathfrak C_N$.  The inverse map is given by taking a closed subset $S$ of $\partial X_A$ consisting of at least two points and putting $T\in \mathfrak T_A$ to be the convex hull of $S$ in $X_A$. Thus elements of $\gcn$ can be thought of as locally finite $F_N$-invariant measures on $\mathfrak T_A$.  Let $\mathfrak T_{A,1}$ be the space of all $Y\in \mathfrak T_A$ such that the element $1\in F_N$ is a vertex of $Y$; we think of $1$ as a base-vertex for every tree $T\in \mathfrak T_{A,1}$.  Then  $\mathfrak T_{A,1}$ is compact subset of $\mathfrak T_A$ and one can further identify $\gcn$ with the set of all finite Borel measures on $\mathfrak T_{A,1}$ which are invariant with respect to ``root change".  This point of view connects  $\gcn$  with the study of ``invariant random subgroups" (IRS) and of ``unimodular graph messures", \cite{AGV12,Bow12,BGK,DDMN,D02,Gr11,Sa11,Vershik,Ve11}. See \cite{KN3} for a more detailed discussion regarding these connections.

The standard topology on $\mathfrak T_A$ (and thus on $\mathfrak C_N$) can be described in terms of suitable ``cylinder" sets.
For any finite non-degenerate subtree $K$ of $X_A$ the ``cylinder"
$Cyl_A(K)$ is defined to be the set of all $Y\in \mathfrak T_A$ such
that $K\subseteq Y$ and such that for every $\xi\in \partial T$ there
exists a terminal edge $e$ of $K$ (oriented ``from'' $K$) such that the geodesic ray from the initial vertex of $e$ to $\xi$ in $X_A$ starts with $e$.  We denote by $\mathcal B_A$ the set of all finite non-degenerate subtrees of $X_A$ and by $\mathbf B_A$ the set of all $F_N$-translation classes $[K]$ of trees $K\in \mathcal B_A$.  The cylinders $Cyl_A(K)$, where $K\in \mathcal B_A$, are compact and open in $\mathfrak T_A$ and they form a basis for the standard topology on $\mathfrak T_A$ (and hence, via the identification of $\mathfrak T_A$ with $\mathfrak C_N$, of the standard topology on $\mathfrak C_N$).
 If $\mu\in\gcn$, and $K\in \mathcal B_A$, then, by $F_N$-invariance of $\mu$, the value $\mu(Cyl_A(K))$ depends only on $\mu$ and the $F_N$-translation class $[K]\in \mathbf B_A$ of $K$. 
 For $K\in \mathcal B_A$ and $\mu\in\gcn$ we define the ``weights" $\langle [K],\mu\rangle_A=\langle K,\mu\rangle_A:=\mu(Cyl_A(K))$. Then any $\mu\in \gcn$ is uniquely determined by its collection of weights $(\langle K,\mu\rangle_A)_{K\in \mathcal B_A}$. There are natural disjoint union ``splitting" formulas for the cylinders which, in view of finite additivity of generalized currents, imply corresponding ``switch condition" equations that are satisfied by the weights of subset currents. These switch conditions have somewhat unexpected form, and that's why we discuss them here in more detail. Let $K\subseteq X_A$ be a finite non-degenerate tree and let $e$ be a terminal (oriented) edge of $K$. Let $q(e)$ be the set of all edges $e'$ of $X_A$ such that $ee'$ is a reduced edge-path in $X_A$. Since $X_A$ is a $(2N)$-regular tree, we have $\#(q(e))=2N-1$.  Denote by $P_+(q(e))$ the set of all nonempty subsets of $q(e)$. Then $Cyl_A(K)=\sqcup_{K'\in P_+(q(e))} Cyl_A(K\cup K')$. Therefore for every $\mu\in \gcn$, for every $K\in \mathcal B_A$ and every terminal edge $e$ of $K$ we have
 \[
 \langle K, \mu\rangle_A=\sum_{K'\in P_+(q(e))} \langle K\cup K', \mu\rangle_A.\tag{!}
 \] 
Equations (!) for elements of $\gcn$ play the role of the switch conditions $(\dag)$ for shift-invariant measures on a two-sided shift space discussed above.  As noted earlier, if $H\le F_N$ is a nontrivial finitely generated subgroup, then the counting subset current $\eta_H\in \gcn$ can be described more explicitly in combinatorial terms. Namely let $\widehat X_A$ be the cover of $\Gamma_A$ corresponding to $H$ and let $\Delta_H\subseteq \widehat X_A$ be the \emph{core} of $\widehat X_A$ that is, the smallest connected subgraph of $\widehat X_A$ whose inclusion into $\widehat X_A$ is a homotopy equivalence. Then the oriented edges of $\Delta_H$ are naturally labelled by elements of $A^{\pm 1}$ and this labelling makes $\Delta_H$ into a \emph{folded $A$-graph} in the sense of the theory of Stallings folds~\cite{KM,St}. That is, for any vertex of $\Delta_H$ and any letter $a\in A^{\pm 1}$ there is at most one edge labelled $a$ in $\Delta_H$ originating from this vertex. Moreover, $\Delta_H$ is also ``cyclically reduced", that is, every vertex of this graph has degree $\ge 2$.  There is a natural bijective correspondence between the set of conjugacy classes of finitely generated nontrivial subgroups of $F_N$ and the set of labelled isomorphism types of finite connected cyclically reduced folded $A$-graphs.

Given any folded (and possibly disconnected) cyclically reduced $A$-graph $\Delta$ and a finite tree $K\in \mathcal B_A$, an \emph{occurrence} of $K$ in $\Delta$ is a label-preserving graph map $f:K\to \Delta$ (sending vertices to vertices and edges to edges, respecting labels) which is an immersion and which is a local homeomorphism at every point of $K$ other than terminal vertices of $K$. Thus if $x$ is a non-terminal vertex of $K$ then $f$ maps a small neighborhood of $x$ in $K$ homeomorphically onto a small neighborhood of $f(x)$ in $\Delta$. In particular, it follows that the degree of $x$ in $K$ is equal to the degree of $f(x)$ in $\Delta$.  We denote by $\langle K,\Delta\rangle_A$ the number of all occurrences of $K$ in $\Delta$. Then it turns out that for every nontrivial finitely generated subgroup $H\le F_N$ and every $K\in \mathcal B_A$ we have $\langle K,\eta_H\rangle_A=\langle K,\Delta_H\rangle_A$.

In order to approximate elements of $\gcn$ by rational subset currents we need to understand how to realize finite collections of nonnegative integral weights (where instead of all trees $K\in\mathcal B_A$ we take a suitable finite subset of $\mathcal B_A$) satisfying the switch conditions (!) by finite cyclically reduced graphs $\Delta$.

Theorem~\ref{thm:A} below (c.f.  Theorem~\ref{thm:realization}) provides the requisite realizability result. 
Theorem~\ref{thm:A} is stated for an arbitrary marking on $F_N$ but the case of the marking $F_N=\pi_1(\Gamma_A)$ given by an ``$N$-rose" corresponding to a free basis $A$ of $F_N$ already conveys the essence of Theorem~\ref{thm:A}.    To set up the relevant notations for the general case, let $N\ge 2$, let $\alpha:F_N\to \Gamma$ be a marking on $F_N$ and let $X=\widetilde \Gamma$.
Let $r\ge 2$ be an arbitrary integer and let $\mathcal B_{\Gamma,r}$
be the set of all finite non-degenerate subtrees $K$ of $X$ such that
for some vertex $p$ of $K$ the distance from $p$ to every terminal
vertex of $K$ is equal to $r$. Let  $\mathcal B'_{\Gamma,r}$ be the
set of all finite non-degenerate subtrees $J$ of $X$ such that for
some edge of $J$  the distance from the midpoint $p$ of this edge to
every terminal vertex of $K$ is equal to $r-\frac{1}{2}$. Now let
$J\in \mathcal B'_{\Gamma,r}$  and let $p$ be the midpoint of an edge $e$ of $J$ such that the distance from $p$ to every terminal vertex of $J$ is equal to $r-\frac{1}{2}$. Let $J_0$ be the connected component of $J-\{p\}$ containing the origin of $e$ and let  $J_1$ be the connected component of $J-\{p\}$ containing the terminus of $e$. 
Let $e_1,\dots, e_n$ be all the terminal edges of $J$ contained in $J_0$ and let $f_1,\dots, f_m$ be all the terminal edges of $J$ contained in $J_1$. We say that $\{e_1,\dots, e_n\}\sqcup \{f_1,\dots, f_m\}$ is the \emph{geometric partition} of the set of terminal edges of $J$.

\begin{theor}[Integral Weight Realization Theorem]\label{thm:A}
Let $N\ge 2$, let $\alpha:F_N\to \Gamma$ be a marking on $F_N$, let $X=\widetilde \Gamma$, and let $r\ge 2$ be an arbitrary integer.

Let $\vartheta:\mathcal B_{\Gamma,r}\to \mathbb Z_{\ge 0}$ be a function such that:

\begin{enumerate}
\item For every $K\in \mathcal B_{\Gamma,r}$ and every $g\in F_N$ we have $\vartheta(gK)=\vartheta(K)$.
\item  For every $J\in \mathcal B'_{\Gamma,r}$ with  the geometric partition $\{e_1,\dots, e_n\}\sqcup \{f_1,\dots, f_m\}$ of the set of terminal edges of $J$ we have:
\[
\sum_{(U_1,\dots, U_n)} \vartheta(J\cup U_1\cup\dots \cup U_n) =\sum_{(V_1,\dots, V_m)}  \vartheta(J\cup V_1\cup\dots \cup V_m)
\]
Here the first sum is taken over all $(U_1,\dots, U_n)$ such that 
$U_i\in P_+(q(e_i))$, and the second sum is taken over all $(V_1,\dots, V_m)$ such that 
$V_j\in P_+(q(f_j))$.

\item There exists $K\in \mathcal B_{\Gamma,r}$ such that $\vartheta(K)>0$. 
\end{enumerate}
Then there exists a finite folded (possibly disconnected) cyclically reduced $\Gamma$-graph $\Delta$ such that for every $K\in \mathcal B_{\Gamma,r}$ we have
\[
\vartheta(K)=\langle K,\Delta\rangle_\Gamma.
\]
\end{theor}

Note that in part (2) of  Theorem~\ref{thm:A} constructing $J\cup U_1\cup\cdots\cup U_n$ amounts to adding new edges to each terminal vertex of $J$ on one side of the geometric partition, while $J\cup V_1\cup\cdots\cup V_m$ is obtained from $J$ by adding edges to the terminal vertices of the other side of the geometric partition.
Note that also if $J$ is as in part (2) the theorem, then for every $(U_1,\dots,U_m)$ and every $(V_1,\dots,V_m)$ as in the theorem we automatically have 
 that the trees $J\cup U_1\cup\dots \cup U_n$  and $J\cup V_1\cup\dots \cup V_m$ belong to $\mathcal B_{\Gamma,r}$.
Theorem~\ref{thm:A} is a substitute for Proposition~\ref{prop:iw} in
the context of subset currents. However, compared to the context of
the two-sided shift $A^\Z$ considered in Proposition~\ref{prop:iw},
even the statement of Theorem~\ref{thm:A} is considerably more
complicated and it requires a significant change of perspective
to properly account for the ``non-linear" nature of the tree
$X=\widetilde \Gamma$ which replaces $\Z$ here. For the same reason
the proof of Theorem~\ref{thm:A} does not rely on Euler circuit
considerations, but instead uses certain kinds of perfect matching
arguments. Crucially, the proof of Theorem~\ref{thm:A} is
constructive, and therefore it can be used to produce explicit
approximations of various versions of ``random" or ``uniform" subset
currents considered in Section~10 of \cite{KN3} by rational subset
currents (see further discussion below).

Using Theorem~\ref{thm:A}  we obtain a new  proof of one of the main results of \cite{KN3}:
\begin{theor}\label{thm:main}
Let $N\ge 2$ be an integer. Then $\gcnr$  is a dense subset of $\gcn$.
\end{theor}
Theorem~\ref{thm:main} generalizes a similar result~\cite{Ka2,Martin} for $\Curr(F_N)$, but the case of $\gcn$ is considerably more difficult. The proof of Theorem~\ref{thm:main} in \cite{KN3} is indirect and relies on deep work of Bowen and Elek about ``unimodular graph measures", that is, measures on spaces of rooted graphs that are invariant, in the appropriate sense, with respect to root-change. Given a free basis $A$ of $F_N$ and the Cayley graph $X_A$ of $F_N$ with respect to $A$, in \cite{KN3} we relate subset currents to root-change invarinat measures on the space $\mathcal T_1(X_A)$ of all infinite subtrees $Y$ of $X_A$ without degree-one vertices such that $Y$ contains the vertex $1$ of $X_A$. For studying  $\mathcal T_1(X_A)$  one can use the results of Bowen~\cite{Bow03,Bow09} and Elek~\cite{Elek}  about weakly approximating these measures by sequences of finite graphs and eventually conclude that $\gcnr$  is dense in $\gcn$. Here we give a direct proof of Theorem~\ref{thm:main}, bypassing the  ``unimodular graph measures" results. The proof shares some similarities with the approaches of Elek and Bowen, but is more combinatorial and explicit.  As another application of Theorem~\ref{thm:A} we solve Problem~10.11 from \cite{KN3}  and obtain:

\begin{theor}[c.f. Theorem~\ref{thm:int}]\label{thm:C}
Let $N\ge 2$ let $\alpha:F_N\to \Gamma$ be a marking on $F_N$ and let $X=\widetilde \Gamma$.  Let $\mu\in \gcn$ be a nonzero subset current such that for every $K\in \mathcal B_\Gamma$ we have $\langle K,\mu\rangle_\Gamma\in \mathbb Z$.

Then there exist $k\ge 1$ and  nontrivial finitely generated subgroups $H_1,\dots, H_k\le F_N$ such that
\[
\mu=\eta_{H_1}+\dots +\eta_{H_k}.
\]

\end{theor}

One of the main reasons for writing this paper was to provide a proof
of Theorem~\ref{thm:main} that is complete and relatively easily understandable by
the $Out(F_N)$ community. The technology developed here has already
proved useful in a new paper of Sasaki~\cite{Sas} who solved a problem
posed in \cite{KN3} and related subset
currents to the Strengthened Hanna Neumann Conjecture.  Other potential
applications include, for example, studying the ``generic volume
distortion factors'' for free group automorphisms.  In Section~9 of \cite{KN3} we
constructed several versions of ``uniform subset currents'' on $F_N$
corresponding to a free basis $A$ of $F_N$, including the ``absolute
uniform current'' $m_A^{\mathcal S}\in\gcn$.  If $T_A\in\cvn$ is the
Cayley tree of $F_N$ with respect to $A$ and if $\phi\in\Out(F_N)$, it
should be possible to interpret the geometric intersection number (as
defined in \cite{KN3}) $\langle T_A, \phi m_A^{\mathcal S}\rangle$  as
the ``generic volume distortion''
$\vol(\Delta_A(\phi(H)))/\vol(\Delta_A(H))$. Here $H$ is a ``random'',
in the sense of projectively approximating $m_A^{\mathcal S}$ finitely
generated subgroup of $F_N$, and $\Delta_A(H)$ is the core Stallings
subgroup graph for $H$ with respect to $A$. In order to provide such a
characterization of  $\langle T_A, \phi m_A^{\mathcal S}\rangle$ one
needs to create a random process, at step $n$ outputting a subgroup
$H_n\le F_N$ such that almost surely $\lim_{n\to\infty}
[\eta_{H_n}]=[m_A^{\mathcal S}]$ in $\pgcn$. (Here for a subset
current $\mu\in \gcn$ we denote by $[\mu]\in \pgcn$ the
\emph{projective class} of $\mu$.) To do that one needs an
explicit procedure for how to projectively approximate $m_A^{\mathcal S}$
by rational subset currents, and Theorem~\ref{thm:A} provides such a procedure.

In addition, there is work in progress by Dounnu Sasaki on developing
the theory of subset currents for surface groups. Obtaining an analog
of Theorem~\ref{thm:main} remains an open problem in that context, and
we hope that an explicit proof of Theorem~\ref{thm:main} for free
groups may prove useful there.

I am particularly grateful to the referee for the careful reading of
the paper and for the detailed and helpful suggestions.

\section{Background}\label{sec:background}

We will use the same notations, conventions and definitions as in \cite{KN3} and only briefly recall some of them here.
If $Y$ is a graph, we denote by $EY$ the set of oriented edges of $Y$. For $e\in EY$ $o(e)$ is the initial vertex of $e$, $t(e)$ is the terminal vertex of $e$ and $e^{-1}\in EY$ is the inverse edge of $e$.

\subsection{The space $\mathfrak C_N$}

Let $F_N$ be a free group of finite rank $N\ge 2$. The space $\mathfrak C_N$ consists of all closes subsets $S\subseteq \partial F_N$ such that $S$ consists of at least two points.
We topologize $\mathfrak C_N$ by choosing a visual metric $d$ on $\partial F_N$ and then using the Hausdorff distance between closed subsets of $\partial F_N$ to metrize $\mathfrak C_N$.  This metric topology on $\mathfrak C_N$ does not depend on the choice of a visual metric on $\partial F_N$ and turns $\mathfrak C_N$ into a locally compact totally disconnected Hausdorff topological space.  The topology on $\mathfrak C_N$ can be described more explicitly in terms of the ``subset cylinders". 
Given a marking $\alpha: F_N\isom \pi_1(\Gamma)$ (where $\Gamma$ is a finite connected graph without degree-one and degree-two vertices), let $X=\widetilde \Gamma$, taken with the simplicial metric, where every edge has length $1$. Then $\alpha$ induces a quasi-isometry between $F_N$ and $X$ and hence gives an identification, via an $F_N$-equivariant homeomorphism, between $\partial F_N$ and $\partial X$. As in~\cite{KN3}, we denote by $\mathcal K_\Gamma$ the set of all finite non-degenerate subtrees $K\subseteq X$.
If $e$ is an oriented edge of $X$, we denote by $Cyl_X(e)$ the set of
all $\xi\in\partial F_N$ such that the geodesic from $o(e)$ to $\xi$
in $X$ starts with $e$.  Thus $Cyl_X(e)\subseteq \partial F_N$ is a
compact-open subset of $\partial F_N$.  Now let $K\in \mathcal
K_\Gamma$.  Let $e_1,\dots, e_n\in EX$ be all the terminal edges of
$K$ (oriented ``from'' $K$).
We define the \emph{subset cylinder} $\mathcal SCyl_\alpha(K)\subseteq \mathfrak C_N$ as the set of all $S\in \mathfrak C_N$ such that $S\subseteq \cup_{i=1}^n Cyl_X(e_i)$ and such that for each $i=1,\dots, n$ $S\cap Cyl_X(e_i)\ne \emptyset$. Then $\mathcal SCyl_\alpha(K)$ is a compact-open subset of $\mathfrak C_N$ and the family $\{\mathcal SCyl_\alpha(K)| K\in \mathcal K_\Gamma\}$ forms a basis for the topology on $\mathfrak C_N$ defined above.

Denote by $q(e)$ the set of all oriented edges $e'$ in $X$ such that $e,e'$ is a reduced edge-path in $X$.  For any set $B$ we denote by $P_+(B)$ the set of all nonempty subsets of $B$.  The following basic fact plays a key role in the theory of subset currents:

\begin{lem}\label{lem:disj}[c.f. Lemma~3.5 in \cite{KN3}]
Let $K\in \mathcal K_\Gamma$ and let $e_1,\dots, e_n$ be all the terminal edges of $K$. 
Then for every $i=1,\dots, n$ we have $\displaystyle \mathcal SCyl_\alpha(K)=\sqcup_{U\in P_+(q(e_i))} \mathcal SCyl_\alpha(K\cup U)$.
\end{lem}

\subsection{Subset currents}

A \emph{subset current} on $F_N$ is a positive Borel measure measure $\mu$ on $\mathfrak C_N$ which is $F_N$-invariant and locally finite, that is, finite on all compact subsets of $\mathfrak C_N$.

The set of all subset currents on $F_N$ is denoted $\gcn$. The space $\gcn$ is endowed with the natural weak-* topology of point-wise convergence of integrals of continuous functions. The weak-* topology on $\gcn$ can be described in more concrete terms:

Let $\mu, \mu_n\in \gcn$.  Then $\lim_{n\to\infty} \mu_n=\mu$ in $\gcn$ if and only if for every finite non-degenerate subtree $K$ of $X$ we have
\[
\lim_{n\to\infty} \mu_n(\mathcal SCyl_\alpha(K))= \mu(\mathcal SCyl_\alpha(K)).
\]

For $K\in \mathcal K_\Gamma$ and $\mu\in\gcn$ denote $\langle K, \mu\rangle_\alpha:=mu(\mathcal SCyl_\alpha(K))$ and call this quantity the \emph{weight} of $K$ in $\mu$.
If $\mu,\mu'\in \gcn$ satisfy $\langle K, \mu\rangle_\alpha=\langle K, \mu'\rangle_\alpha$ for all $K\in \mathcal K_\Gamma$, then $\mu=\mu'$. Note that if $K\in \mathcal K_\Gamma$ and $g\in F_N$ then $g\mathcal SCyl_\alpha(K)=\mathcal SCyl_\alpha(gK)$. Hence for any $\mu\in\gcn$, $g\in F_N$ and $K\in\mathcal K_\Gamma$ we have $\mu(\mathcal SCyl_\alpha(K))=\mu(g\mathcal SCyl_\alpha(K))$, so that $\langle K, \mu\rangle_\alpha=\langle gK, \mu\rangle_\alpha$.
For a given finite subtree $K$ of $X$, we denote the $F_N$-translation class of $K$ by $[K]$ (so that $[K]$ consists of all the translates of $K$ by elements of $F_N$). 
We put $\langle [K], \mu\rangle_\alpha:=\langle K, \mu\rangle_\alpha$
and call it the \emph{weight} of $[K]$ in $\mu$.

Lemma~\ref{lem:disj} immediately implies (c.f. Proposition~3.11 in \cite{KN3}):
\begin{prop}[Kirchhoff formulas for weights]\label{prop:kirch}
Let $K$ be a finite non-degenerate subtree of $X$. Let $e$ be one of the terminal edges of $K$ and let $\mu\in \gcn$.
Then
\[
\langle K, \mu\rangle_\alpha=\sum_{U\in P_+(q(e))} \langle K\cup U, \mu\rangle_\alpha.\tag{$\bigstar$}
\]

\end{prop}

\subsection{$\Gamma$-graphs}

Let $\alpha: F_N\isom \pi_1(\Gamma)$ be a marking. A \emph{$\Gamma$-graph} is a graph $\Delta$
together with a graph morphism $\tau:\Delta\to\Gamma$. For a vertex $x\in V\Delta$ we sat that the \emph{type} of $x$ is the
vertex $\tau(x)\in V\Gamma$. Similarly, for an oriented edge $e\in
E\Gamma$ the \emph{type} of $e$, or the \emph{label} of $e$ is the
edge $\tau(e)$ of $\Gamma$. Every covering of $\Gamma$ has a canonical $\Gamma$-graph
structure. In particular, $\Gamma$ itself is a $\Gamma$-graph and so
is the universal cover $\widetilde\Gamma$ of $\Gamma$. Also, every
subgraph of a $\Gamma$-graph is again a $\Gamma$-graph.

Let $\tau_1:\Delta_1\to\Gamma$ and $\tau_2:\Delta_2\to\Gamma$ be $\Gamma$-graphs. A graph-map
$f:\Delta_1\to\Delta_2$ is called a \emph{$\Gamma$-map}, or
\emph{$\Gamma$-morphism}, if it respects the labels of vertices and
edges, that is if $\tau_1=\tau_2\circ f$. A $\Gamma$-graph $\Delta$ is \emph{folded} if the labeling map  $\tau:\Delta\to\Gamma$ is an immersion, that is, if $\tau$ is locally injective.

\begin{defn}[Link of a vertex]
Let $\Delta$ be a $\Gamma$-graph. For a vertex $x\in V\Delta$ denote
by $Lk_\Delta(x)$ (or just by $Lk(x)$) the function
\[
Lk_\Delta(x): E\Gamma\to\mathbb Z_{\ge 0}
\]
where for every $e\in E\Gamma$ the value
$\left(Lk_\Delta(x)\right)(e)$ is the number of edges of $\Delta$ with
origin $x$ and label $e$.
\end{defn}

Thus a $\Gamma$-graph $\Delta$ is folded if and only if for every
vertex $x\in V\Delta$ and every $e\in E\Gamma$ we have
\[
\left(Lk_\Delta(x)\right)(e)\le 1.
\]

If $\Delta$ is folded, we will also think of $Lk_\Delta(x)$ as a
subset of $E\Gamma$ consisting of all those $e\in E\Gamma$ with
$\left(Lk_\Delta(x)\right)(e)=1$, that is, of all $e\in E\Gamma$ such
that there is an edge in $\Delta$ with origin $x$ and label $e$.

We say that a nonempty finite $\Gamma$-graph $\Delta$ is \emph{cyclically reduced} if $\Delta$ is folded and every vertex of $\Delta$ has degree $\ge 2$. If $\tau:\Delta\to \Gamma$ is a cyclically reduced $\Gamma$-graph, then $W:=\tau_\#(\pi_1(\Delta))\le \pi_1(\Gamma)$ is a finitely generated subgroup of $\pi_1(\Gamma)$. Recall that we also have a marking $\alpha: F_N\isom \pi_1(\Gamma)$. We say that the subgroup $H:=\alpha^{-1}(W)\le F_N$ is \emph{represented} by $\Delta$. The conjugacy class of $[H]$ in $F_N$ does not change if we replace $\alpha$ by an equivalent marking.


\begin{defn}[Occurrence]\label{defn:occur}
Let $K\subseteq \widetilde \Gamma$ be a finite non-degenerate subtree
(recall that $\widetilde \Gamma$ and all of its subgraphs have
canonical $\Gamma$-graph structure).

Let $\Delta$ be a finite cyclically reduced  $\Gamma$-graph. 
An \emph{occurrence} of $K$ in $\Delta$ is a $\Gamma$-morphism
$\mathfrak O:K\to\Delta$ such that for every vertex $x$ of $K$ of degree at
least $2$ in $K$ we have $Lk_K(x)=Lk_\Delta(\mathfrak O(x))$.

We denote the number of all occurrences of $K$ in $\Delta$ by $\langle
K; \Delta\rangle_\Gamma$, or just $\langle
K; \Delta\rangle$.
\end{defn}

In topological terms, a $\Gamma$-morphism $\mathfrak O:K\to\Delta$ is an occurrence
of $K$ in $\Delta$ if $\mathfrak O$ is an immersion and if $\mathfrak O$ is a covering map
at every point $x\in K$ (including interior points of edges) except
for the degree-1 vertices of $K$. That is, for every $x\in K$, other
than a degree-1 vertex of $K$, $\mathfrak O$ maps a small neighborhood of $x$ in
$K$ homeomorphically \emph{onto} a small neighborhood of $\mathfrak O(x)$ in
$\Delta$. We need the following key fact from~\cite{KN3}:

\begin{propdfn}\label{prop:wd}
Let $\alpha:F_N\to\pi_1(\Gamma)$ be a marking on $F_N$ and let
$X=\widetilde\Gamma$. Recall that $\mathcal K_\Gamma$ is the set of
all non-degenerate finite simplicial subtrees of $X$.
Let $\tau:\Delta \to\Gamma$ be a finite cyclically reduced $\Gamma$-graph.

Then there is a unique generalized
current $\mu_\Delta\in\gcn$ such that for every $K\in \mathcal
K_\Delta$
\[
\langle K, \mu_\Delta\rangle_\alpha=\langle K;\Delta\rangle_\Gamma
\]
Moreover, if $\Delta$ is also connected, then $\mu_\Delta=\eta_H$, where $H\le F_N$ is the finitely generated subgroup of $F_N$ represented by $\Delta$.
\end{propdfn}

\section{More on cylinders and Kirchhoff-type formulas}

\begin{conv}
From now and for the remainder of this paper, unless specified
otherwise, we fix a marking $\alpha: F_N\to \pi_1(\Gamma)$. Put
$X=\widetilde \Gamma$. We also equip $X$ with the simplicial metric
$d$, by giving each edge of $X$ length $1$.
\end{conv}

Let $K\subseteq X$ be a nondegenerate finite subtree and let $e$ be a
terminal edge of $K$. For an integer $m\ge 1$ we say that a finite
nondegenerate subtree $U\subseteq X$ is \emph{$(K,e,m)$-admissible}
if:
\begin{enumerate}
\item We have $K\cap U=\{t(e)\}$.
\item For every terminal vertex $v$ of $U$ such that $v\ne t(e)$ we
  have $d(t(e),v)=m$.
\end{enumerate}
For $m=0$ we also say that the degenerate tree $U=\{t(e)\}$ is
\emph{$(K,e,0)$-admissible}.

For $m\ge 1$ we denote by $\mathcal B(K,e,m)$ the set of all $U$ such that $U$ is
$(K,e,m)$-admissible. Thus  $P_+(q(e))=\mathcal B(K,e,1)$. Lemma~\ref{lem:disj} easily implies:

\begin{cor}\label{cor:m-disj}
Let  $K\subseteq X$ be a nondegenerate finite subtree and let $e$ be a
terminal edge of $K$. Then for every integer $m\ge 1$ we have
\[
\mathcal SCyl_\alpha(K)=\sqcup_{U\in \mathcal B(K,e,m)} \mathcal SCyl_\alpha(K\cup U).
\]
\end{cor}

\begin{defn}[Round graph]
For an integer $r\ge 1$, we say that a finite subtree $K$ of $X$ is a
\emph{round graph} of \emph{grade $r$} in $X$ if there exists a (necessarily unique) vertex
$v$ of $K$ such that for every terminal vertex $u$ of $K$ we have
$d(v,u)=r$. 
\end{defn}

Let $K\subseteq X$ be a nondegenerate finite subtree and let $v$ be a
vertex of $K$ (possibly a terminal vertex). We denote by $R(K,v)$ the maximum of $d(v,v')$ where
$v'$ varies over all terminal vertices of $K$. The fact that $K$ is
nondegenerate means that $R(K,v)\ge 1$. 

Let $e_1,\dots, e_n$ be the terminal edges of $K$ and let $r\ge
R(K,v)$ be an integer. We say that an $n$-tuple $\mathcal
T=(U_1,\dots,U_n)$ of finite subtrees $U_i$ of $X$ is
\emph{$(K,v,r)$-admissible} if for each $i=1,\dots, n$ the tree $U_i$
is $(K,e_i,m_i)$-admissible, where $m_i=r-d(v,t(e_i))$. 
Note that if $\mathcal
T=(U_1,\dots,U_n)$ is $(K,v,r)$-admissible and $K'=K\cup
U_1\dots \cup U_n$ then for every terminal vertex $u$ of
$K'$ we have $d(v,u)=r$. Thus $K'$ is a round graph of grade $r$ with
center $v$. Corollary~\ref{cor:m-disj} directly implies:

\begin{cor}\label{cor:r-disj}
Let $K\subseteq X$ be a nondegenerate subtree with terminal edges
$e_1,\dots, e_n$. Let $v$ be a vertex of $K$ and let $r\ge R(K,v)$ be
an integer. Denote by $\mathcal B(K,v,r)$ the set of all
$(K,v,r)$-admissible $n$-tuples. Then 
\[
\mathcal SCyl_\alpha(K)=\underset{(U_1,\dots,U_n)\in \mathcal B(K,v,r)}{\bigsqcup} \mathcal SCyl_\alpha(K\cup
U_1\cup\dots \cup U_n) 
\]
\end{cor}

For a finite nondegenerate subtree $K\subseteq X$ we put
$r(K)$ to be the mimimum of $R(K,v)$ where $v$ varies over all
vertices of $K$. We refer to $r(K)$ as the \emph{radius} of $K$.

In view of finite additivity of subset currents,
Corollary~\ref{cor:m-disj} and Corollary~\ref{cor:r-disj} immediately
imply:

\begin{cor}\label{cor:b}
Let $K\subseteq X$ be a finite non-degenerate subtree of $X$ and let $\mu\in\gcn$. Then:

\begin{enumerate} 
\item For any terminal edge $e$ of $K$ and any integer $m\ge 1$ we
  have
\[
\langle K, \mu\rangle_\alpha=\sum_{U\in \mathcal B(K,e,m)} \langle K\cup U, \mu\rangle_\alpha.
\]
\item Let $v$ be a vertex of $K$, let $e_1,\dots, e_n$ be the terminal
  edges of $K$ and let $r\ge R(K,v)$ be an integer. Then 
  have
\[
\langle K, \mu\rangle_\alpha=\sum_{(U_1,\dots,U_n)\in \mathcal
  B(K,v,r)} \langle K\cup U_1\dots \cup U_n, \mu\rangle_\alpha.
\]
\end{enumerate}
\end{cor}
Recall that, as noted earlier, for part (2) above, if $(U_1,\dots,U_n)\in \mathcal
  B(K,v,r)$ and $K'=K\cup U_1\dots \cup U_n$ then for any
  terminal vertex $u$ of $K'$ we have $d(v,u)=r$, so that
  $K'$ is a round graph of grade $r$ in $X$.
Thus part (2) of Corollary~\ref{cor:b} implies that, for $\mu\in \gcn$
and an integer $r\ge 1$, knowing the $\mu$-weights of all round graphs of
grade $r$ uniquely determines the $\mu$-weights of all the subtrees of radius
$\le r$.

\begin{defn}[Semi-round graph]
Let $p$ be the mid-point of an edge $e$ of $X$ and let $r\ge 2$ be an
integer. We say that a finite subtree $J$ of $X$ is a \emph{semi-round
  graph} of \emph{grade $r$} with \emph{center} $p$ if $e\in J$ and if for every terminal vertex $u$ of $J$ we have $d(p,u)=r-\frac{1}{2}$. Thus for every terminal vertex $u$ of $J$ belonging to the connected component of $J-\{p\}$ containing $o(e)$ we have $d(o(e),v)=r-1$. Similarly, for every terminal vertex $u$ of $J$ belonging to the connected component of $J-\{p\}$ containing $t(e)$ we have $d(t(e),v)=r-1$. 
\end{defn}

\begin{defn}[Child of a round graph]\label{defn:child} 
Let $r\ge 2$ and let $K\subseteq X$ be a round graph of grade $r$ in $X$
centered at a vertex $v$ of $X$. Let $e$ be an edge of $K$ with
$o(e)=v$. Let $p$ be the mid-point of $e$. We define a semi-round
graph of grade $r$, centered at $p$, called the \emph{$e$-child of
  $K$} and denoted $K_e$, as follows:

The graph $K_e$ consists of all points $q\in K$ with $d(p,q)\le r-\frac{1}{2}$. In other words, $K_e$ is obtained from $K$ by removing all those
terminal vertices $u$ of $K$ and the terminal edges of $K$ adjacent to
these vertices such that the geodesic $[v,u]$ does not pass through
the edge $e$.  The definition of $K_e$ is illustrated in Figure~\ref{Fi:child}.
\end{defn}

\begin{figure}

\psfrag{KE}{Child $K_e$}

\psfrag{K}{Round graph $K$}

\includegraphics[scale=.6]{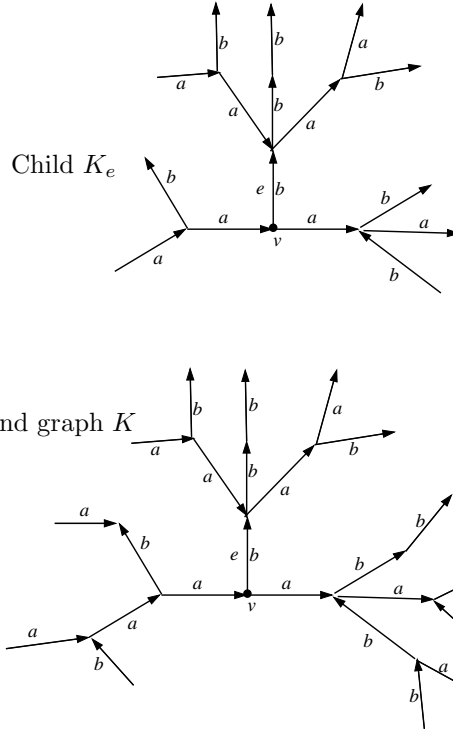}


\caption{Child $K_e$ of a round graph $K$ of grade $3$ with center
  $v$. Here $N=2$, $F_2=F(a,b)$ and $\Gamma$ is the standard rose
  corresponding to the free basis $\{a,b\}$ of $F(a,b)$.}\label{Fi:child}
\end{figure}

If $H$ is a semi-round graph of grade $r$ with center at the midpoint
$p$ of an edge $e$, then $H$ can be enlarged to  round graphs of grade
$r$ in two different ''directions'', namely to round graphs centered at
$o(e)$ and at $t(e)$. This yields the following:

\begin{prop}\label{prop:sum}
Let $J\subseteq X$ be a semi-round graph of grade $r\ge 2$ centered at
the midpoint $p$ of an edge $e$. Let $J_0$ and $J_1$ be the connected
components of $J-\{p\}$ containing $o(e)$ and $t(e)$ accordingly. Let
$e_1,\dots, e_n$ be all the terminal edges of $J$ contained in $J_0$
and let $f_1,\dots, f_k$ be all the terminal edges of $J$ contained in
$J_1$. Let $\mathcal B_0$ be the set of all $n$-tuples of the form
$(U_1,\dots, U_n)$ where each $U_i\in P_+(q(e_i))$, and let $\mathcal B_1$ be the set of all $k$-tuples of the form
$(V_1,\dots, V_k)$ where each $V_j\in P_+(q(f_j))$.

Then for any $\mu\in\gcn$ we have
\begin{gather*}
\langle J, \mu\rangle_\alpha=\sum_{(U_1,\dots, U_n)\in \mathcal B_0} \langle J\cup U_1\dots \cup U_n,
\mu\rangle_\alpha=\sum_{(V_1,\dots, V_k)\in \mathcal B_1} \langle J\cup V_1\dots \cup V_k,
\mu\rangle_\alpha.
\end{gather*}

(Note that in the above summation each $J\cup U_1\dots \cup U_n$ is a
round graph of grade $r$ centered at $o(e)$ and each $J\cup V_1\dots
\cup V_k$ is a round graph of grade $r$ centered at $t(e)$.)

\end{prop}

\begin{proof}
This statement is a direct corollary of Proposition~\ref{prop:kirch}.
\end{proof}
 
\section{Finite-dimensional polyhedral approximations of $\gcn$ and the integral weight realization theorem}

Let $r\ge 2$ be an integer. Denote by $\mathcal B_{\Gamma, r}$ the set
of all finite subtrees $K\subseteq X$ such that $K$ is a round graph
of grade $r$ in $X$. Also, denote by $\mathbf B_{\Gamma, r}$ the set
of all $F_N$-translation classes $[K]$ of trees $K\in \mathcal
B_{\Gamma, r}$.

Denote by $\mathbf J_{\Gamma, r}$ the set of all $F_N$-translation
classes $[J]$ of semi-round graphs $J\subseteq X$ of grade $r$.

\begin{defn}[Approximating polyhedra]\label{defn:Q}

Denote by $\mathcal Q_{\Gamma, r}$ the set of all functions
$\vartheta: \mathcal B_{\Gamma, r}\to \mathbb R_{\ge 0}$ satisfyng the
following properties:

\begin{enumerate}
\item For every $K\in \mathcal B_{\Gamma, r}$ and every $g\in F_N$ we
  have $\vartheta(K)=\vartheta(gK)$.
\item For every semi-round graph $J\subseteq X$ of grade $r$, in the
  notations of Proposition~\ref{prop:sum} we have
\[
\sum_{(U_1,\dots, U_n)\in \mathcal B_0} \vartheta(J\cup U_1\dots \cup U_n)
=\sum_{(V_1,\dots, V_k)\in \mathcal B_1} \vartheta(J\cup V_1\dots \cup V_k).
\]
\end{enumerate}

We call $\mathcal Q_{\Gamma, r}$ an \emph{approximating polyhedron}. 

\end{defn}

Note that since $X$ is locally finite, there are only finitely many
$F_N$-translation classes $[K]$ of trees $K\in \mathcal B_{\Gamma,
  r}$. Thus a point $\theta\in \mathcal Q_{\Gamma, r}$ can be viewed
as a function from a finite set $\mathbf B_{\Gamma, r}$ to $\mathbb
R_{\ge 0}$. Namely, if $m$ is the cardinality of $\mathbf B_{\Gamma,
  r}$, we can view $\mathcal Q_{\Gamma, r}$ as a subset of $\mathbb
R_{\ge 0}^m$, given by finitely many linear equations with integer
coefficients coming from condition (2) in Definition~\ref{defn:Q}.

The following lemma is a straightforward inductive corollary of
Definition~\ref{defn:occur}:

\begin{lem}\label{lem:v}
Let $\Delta$ be a finite cyclically reduced $\Gamma$-graph.
Then for every integer $r\ge 1$
\[
\# V(\Delta)=\sum_{[K]\in \mathbf B_{\Gamma,
  r}} \langle K, \Delta\rangle_\alpha.
\]
\end{lem}

Recall, that, by definition, any $\Gamma$-graph $\Upsilon$ comes equipped with a
``labelling'' graph-map $\tau: \Upsilon\to\Gamma$.

The following statement is Theorem~\ref{thm:A} from the Introduction:

\begin{thm}\label{thm:realization}
Let $r\ge 2$ and let $\vartheta\in  \mathcal Q_{\Gamma, r}$ be such
that for some $K_0\in \mathcal B_{\Gamma, r}$ we have
$\vartheta(K_0)>0$. Suppose also that for every $K\in \mathcal
B_{\Gamma, r}$ we have $\vartheta(K)\in \mathbb Z$. Then there exists
a cyclically reduced (and possibly disconnected) finite $\Gamma$-graph
$\Delta$ such that for every $K\in \mathcal
B_{\Gamma, r}$ we have $\vartheta(K)=\langle K, \Delta\rangle_\alpha$.
\end{thm}
\begin{proof}

For each $[K]\in \mathbf
B_{\Gamma, r}$ we choose a representative $K\in [K]$, so that  $K\in
\mathcal B_{\Gamma, r}$ and let $v=v_{[K]}$
be the center vertex of $K$. Thus $K$ is a round graph of rank $r$
centered at $v$. Denote $n_{[K]}:=\vartheta(K)$. By
assumption every $n_{[K]}\ge 0$ is an integer and there exists
$K_0\in \mathcal B_{\Gamma, r}$ such that $n_{[K_0]}\ge 0$.

For every $[K]\in \mathbf
B_{\Gamma, r}$ we make  $n_{[K]}$ copies $v_{[K],i}$ (where
$i=1,\dots, n_{[K]}$) of the vertex $v_{[K]}$ together with
``half-links'' of $v_{[K]}$ in $K$. That is for each $v_{[K],i}$  and
for each edge $e$ of $K$ with $o(e)=v_{[K]}$ we attach a closed
half-edge $[v_{[K],i}, p_{e,i}]$ at $v_{[K],i}$ representing a copy of
the initial half of the edge $e$.

We refer to the points  $p_{e,i}$ as \emph{sub-vertices} and to the
segments $[v_{[K],i}, p_{e,i}]$ as \emph{sub-edges}.
We endow each sub-vertex $p_{e,i}$ with a \emph{decoration}, which is
an ordered pair $(\tau(e), [K_e])$, where $K_e$ is the $e$-child of $K$ at
$v$. 

Let $\Omega_\vartheta$ be the collection of all the decorated ``half-links" obtained in this way. Thus $\Omega_\vartheta$ consists of $M:=\sum_{[K]\in \mathbf B_{\Gamma, r}} \vartheta(K)=\sum_{[K]\in \mathbf B_{\Gamma, r}}  n_{[K]}$ ``half-links".

Condition~(2) in Definition~\ref{defn:Q} implies that for every
semi-round graph $J\subseteq X$ of grade $r$ with center $p$ being a
mid-point of an oriented edge $e_J$ of $X$, the number of sub-vertices with
decoration  $(\tau(e_J), [J])$ is equal to the number of sub-vertices
with decoration  $(\tau(e_J^{-1}), [J])$.

For each $[J]\in \mathbf J_{\Gamma, r}$ as above we choose a matching
(i.e. a bijection) between the set of subvertices  in $\Omega_\vartheta$ with
decoration  $(\tau(e_J), [J])$ and the set of sub-vertices
with decoration  $(\tau(e_J^{-1}), [J])$

We then identify each sub-vertex with decoration  $(\tau(e_J), [J])$
with the corresponding to it under this matching subvertex with
decoration  $(\tau(e_J^{-1}), [J])$. We perform these identifications
simultaneously for all $[J]\in \mathbf J_{\Gamma, r}$. This gluing procedure is illustrated in Figure~\ref{Fi:glue}.
\begin{figure}

\psfrag{HV}{``Half-link'' of $v$ in $K$}
\psfrag{HV'}{``Half-link'' of $v'$ in $K'$}
\psfrag{LD}{Local picture in $\Delta$}
\psfrag{aJ}{$(a,[J])$}
\psfrag{a1J}{$(a^{-1},[J])$}
\psfrag{K}{$K$}
\psfrag{K'}{$K'$}

\includegraphics[scale=.8]{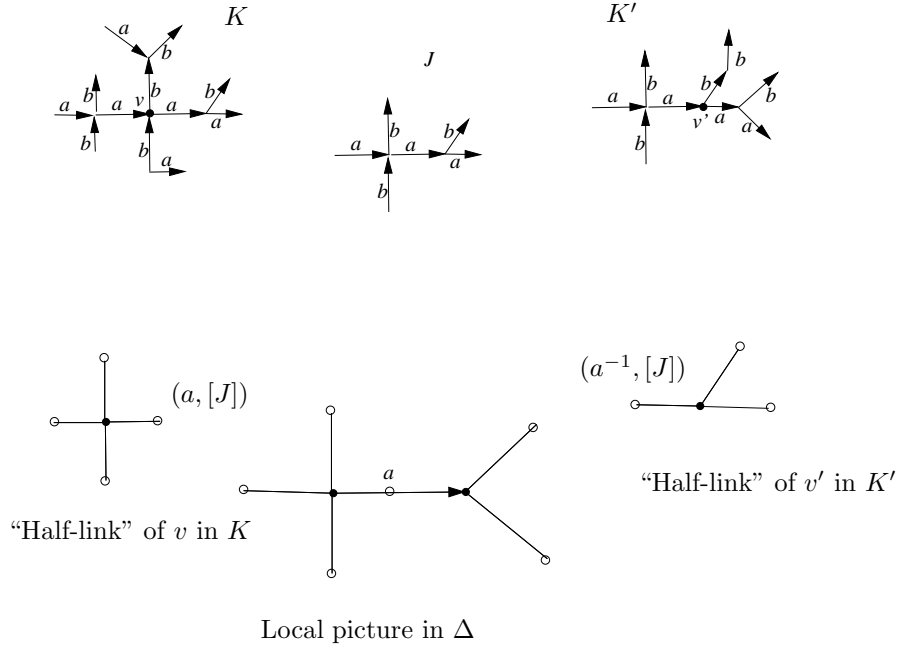}

\caption{Illustration of the ``gluing'' procedure for constructing
  $\Delta$ in the proof of Theorem~\ref{thm:realization}.
 Here $N=2$, $F_2=F(a,b)$ and $\Gamma$ is the standard rose
  corresponding to the free basis $\{a,b\}$ of $F(a,b)$. Filled-in
  circles represent vertices and non-filled circles represent ``sub-vertices''.}\label{Fi:glue}
\end{figure}

The resulting object $\Delta$ has a natural structure of a graph,
where every vertex is of the form  $v_{[K],i}$ and every edge is
obtained by gluing two sub-edges along a sub-vertex; thus subvertices
become mid-points of edges in $\Delta$. Moreover, $\Delta$ inherits a natural $\Gamma$-graph structure as
well. Indeed, an oriented edge $f$ in $\Delta$ arises as the result of gluing
a sub-edge $[v_{[K],i}, p_{e,i}]$ and a sub-edge $[v_{[K'],j},
p_{e',j}]$ by indentifying the sub-vertices  $p_{e,i}$ and $p_{e',j}$
where $p_{e,i}$ is decorated by $(\tau(e), [K_e])$ and  $p_{e',j}$ is
decorated by $(\tau(e'), [K_{e'}])$ such that  $[K_e]=[K_{e'}]$ and
such that $\tau(e')=\tau(e)^{-1}$. In $\Delta$ we have
$o(f)=v_{[K],i}$ and $t(f)=v_{[K'],j}$. We put $\tau(f):=\tau(e)\in
E\Gamma$ and $\tau(f^{-1}):=\tau(e')=\tau(e)^{-1}$.  Also, for each vertex $v_{[K],i}$ of $\Delta$ put $\tau(v_{[K],i}):=\tau(v_{[K]})$.

This turns
$\Delta$ into a nonempty $\Gamma$-graph. Moreover, by construction $\Delta$ is
finite, folded and cyclically reduced and the number of vertices in
$\Delta$ is equal to $M=\displaystyle\sum_{[K]\in \mathbf
B_{\Gamma, r}} \vartheta(K)=\sum_{[K]\in \mathbf B_{\Gamma, r}}  n_{[K]}$.

Note that by construction, for every vertex  $v_{[K],i}$ of $\Delta$
we have $Lk_\Delta(v_{[K],i})=Lk_K (v_{[K]})$. (Recall that $v_{[K]}$ is
the center vertex of the round graph $K$).
Moreover, by definition of a child of a round graph and using the fact that $r\ge
2$ we see that if $f=[v_{[K],i},v_{[K'],j}]$ is an edge of $\Delta$ as
in the preceding paragraph, then
$Lk_\Delta(v_{[K'],j})=Lk_K(t(e))$. Iteratively applying this crucial
fact to the spheres of increasing radius around the center vertex in $K$, we see that for each vertex $v_{[K],i}$ of $\Delta$ as above,
sending $v_{[K]}$ to $v_{[K],i}$ extends to a (necessarily unique)
morphism of $\Gamma$-graphs $\mathfrak O:K\to \Delta$ with $\mathfrak
O(v_{[K]})=v_{[K],i}$ such that $\mathfrak O$ is an occurrence of $K$
in $\Delta$ in the sense of Definition~\ref{defn:occur}.

Moreover, given a vertex $u$ of $\Delta$, there exists exactly one
occurrence of a round graph of grade $r$ in $\Delta$ that sends the
center of that round graph to $u$ (this occurrence corresponds to
taking the ball of radius $r$ in $\widetilde \Delta$ centered at a
lift of $u$). Thus, by construction, we see that for every 
 $[K]\in \mathbf
B_{\Gamma, r}$ the number of occurrences of $[K]$ in $\Delta$ is equal
to $n_{[K]}$. Hence for every $K\in \mathcal
B_{\Gamma, r}$ we have $\vartheta(K)=\langle K, \Delta\rangle_\alpha$,
as required.

\end{proof}

\begin{rem}\label{rem:alt}
There is an alternative equivalent description of the graph $\Delta$
constructed in the proof of Theorem~\ref{thm:realization}.  Namely, for every $[K]\in \mathbf
B_{\Gamma, r}$ we make  $n_{[K]}=\vartheta(K)$ copies $[K]_i$ (where
$i=1,\dots, n_{[K]}$) of $K$ and denote the center vertex of $[K]_i$ by 
$v_{[K],i}$. We then look at the set $\Xi$ of all pairs $([K]_i,e)$ where $[K]_i$ is as above
and $e$ is an edge of $K$ with $o(e)=v_{[K]}$, the center vertex of
$K$. We endow each  $([K]_i,e)$ with a ``decoration'' $(\tau(e),
[K_e])$. Thus $[K_e]$ is a semi-round graph of grade $r$, which comes
from the ball of radius $r-\frac{1}{2}$ in $K$ centered at the
midpoint of $e$. Condition~(2) in Definition~\ref{defn:Q} implies that for every
semi-round graph $J\subseteq X$ of grade $r$ with center $p$ being a
mid-point of an oriented edge $e_J$ of $X$, the number elements of 
$\Xi$ with decoration $(\tau(e_J), [J])$ is is equal to the number of
elements of $\Xi$ with decoration  $(\tau(e_J^{-1}), [J])$.

For each $[J]\in \mathbf J_{\Gamma, r}$ as above we choose a matching
between the set of elements of 
$\Xi$ with decoration $(\tau(e_J), [J])$ and the set of
elements of $\Xi$ with decoration  $(\tau(e_J^{-1}), [J])$. Then we perform partial gluings on the disjoint union $\Omega$ of all $[K]_i$ (where
$[K]$ varies over $\mathbf B_{\Gamma, r}$) as follows. Whenever
$([K]_i,e)$ is matched with $([K']_j,e')$, it follows that the
$e$-child $K_e$ of $K$ is (canonically) isomorphic as a $\Gamma$-graph
to the $e'$-child $K'_{e'}$ of $K'$. (Recall that $K_e$ is the ball of
radius $r-\frac{1}{2}$ in $K$ centered at the midpoint of $e$ and that
$K'_{e'}$ is the ball of radius $r-\frac{1}{2}$ in $K'$ centered at
the midpoint of $e'$). We glue the copy of $K_e$ in $[K]_i$ to the
copy $K'_{e'}$ in $[K']_j$ along the $\Gamma$-graph isomorphism
between $K_e$ and $K'_{e'}$.  We perform these gluings simultaneously,
on the disjoint union $\Omega$ of all $[K]_i$, as $[K]$ varies over $\mathbf B_{\Gamma,
  r}$. The result is a cyclically reduced finite $\Gamma$-graph which
is the same as the $\Gamma$-graph $\Delta$ constructed in the proof of
Theorem~\ref{thm:realization}.  This alternative  gluing procedure is illustrated in Figure~\ref{Fi:glue1}.
\begin{figure}

\psfrag{K}{$K$}
\psfrag{K'}{$K'$}
\psfrag{J}{$J$}

\psfrag{LD}{Local picture in $\Delta$}

\includegraphics[scale=.8]{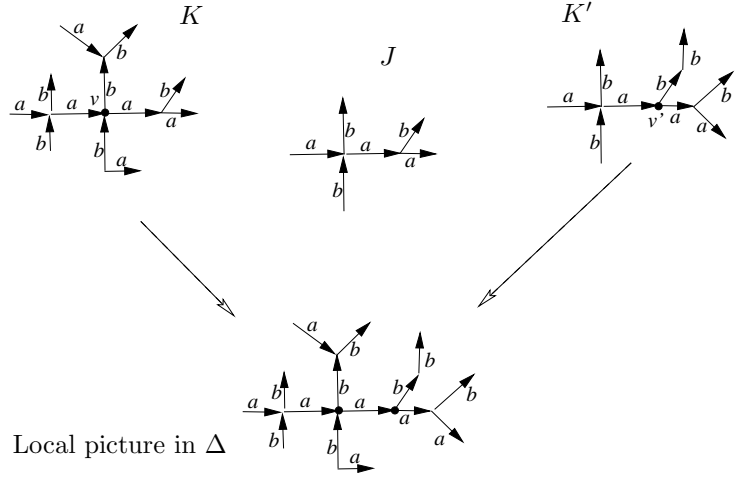}

\caption{Illustration of the alternative ``gluing'' procedure for constructing
  $\Delta$ as in Remark~\ref{rem:alt}..
 Here $N=2$, $F_2=F(a,b)$ and $\Gamma$ is the standard rose
  corresponding to the free basis $\{a,b\}$ of $F(a,b)$.}\label{Fi:glue1}
\end{figure}

\end{rem}

\begin{rem}
Suppose that in Theorem~\ref{thm:realization} $\vartheta\in  \mathcal Q_{\Gamma, r}$ has the property that whenever $\vartheta(K)>0$ then $K\subseteq X$ is a geodesic segment of length $2r$ in $X$. Then for each such $K$ the center vertex $v_{[K]}$ (which is the mid-point of this segment) has degree $2$ in $K$ and the proof of   Theorem~\ref{thm:realization} produces a finite cyclically reduced graph $\Delta$ where every vertex has degree $2$, so that $\Delta$ is a disjoint union of finitely many simplicial circles.  One can use this fact to adapt the proof of Theorem~\ref{thm:dense} below to the case of ordinary geodesic currents and to produce a new proof (different from those given in \cite{Martin,Ka2}) that the set of rational currents is dense in $\Curr(F_N)$.
\end{rem}

The following statement (c.f. Theorem~\ref{thm:C} from the Introduction) provides a positive answer to Problem~10.11 in
\cite{KN3}:

\begin{thm}\label{thm:int}
Let $\mu\in \gcn$ be a nonzero subset current such that for every
nondegenerate finite subtree $K\subseteq X$ we have $\langle K,
\mu\rangle\in \mathbb Z$. Then there exists a finite cyclically
reduced (possibly disconnected) $\Gamma$-graph $\Delta$ such that 
$\mu=\mu_\Delta$. 
\end{thm}

\begin{proof}
Put $M:=\sum_{[K]\in \mathbf B_{\Gamma,
  1}} \langle K, \mu\rangle_\alpha$. For every $r\ge 2$ define the function $\theta_r:  \mathcal B_{\Gamma,
  r}\to \mathbb R_{\ge 0}$ by $\theta_r(K):=\langle K,
\mu\rangle_\alpha$, where $K\in \mathcal B_{\Gamma,
  r}$. Since $\mu$ is a nonzero subset current, we have that $\theta_r\in \mathcal
Q_{\Gamma, r}$ for all $r\ge 2$. 

Hence, by Theorem~\ref{thm:realization}, for every $r\ge 1$ there
exists a finite cyclically reduced $\Gamma$-graph $\Delta_r$ such that
$\langle K, \Delta_r\rangle_\alpha=\langle K, \mu\rangle_\alpha$ for
every $K\in \mathcal B_{\Gamma, r}$. Corollary~\ref{cor:b} then
implies that for every $r\ge 2$ and every finite nondegenerate subtree
$K$ of $X$ of radius $\le r$ we have $\langle K,
\Delta_r\rangle_\alpha=\langle K, \mu\rangle_\alpha$. Hence, by
Lemma~\ref{lem:v}, each graph $\Gamma_r$ has exactly $M$ vertices.
There are only finitely many isomorphism types of finite cyclically
reduced $\Gamma$-graphs with $M$ vertices. Therefore there exists a
finite cyclically reduced $\Gamma$-graph $\Delta$ such that for some
sequence $r_n\to \infty$ as $n\to\infty$ the graph $\Delta_r$ is
isomorphic, as $\Gamma$-graph, to $\Delta$. For any finite  nondegenerate subtree
$K$ of $X$ there exists some $r_n$ such that $r_n$ is $\ge$ the radius
of $K$. Therefore, by construction, for every finite  nondegenerate subtree
$K$ of $X$ we have $\langle K,
\Delta\rangle_\alpha=\langle K, \mu\rangle_\alpha$.
This implies that $\mu_\Delta=\mu$, as required.

\end{proof}

\section{Rational subset currents are dense}

\begin{thm}\label{thm:dense}
Let $N\ge 2$. Then the set $\gcnr$ of all rational subset currents is dense in
$\gcn$.
\end{thm}

\begin{proof}
Let $\mu\in \gcn$ be a nonzero subset current. To show that $\mu$ can be approximated by rational subset currents it
suffices to show that for every integer $r\ge 1$ and any $\epsilon>0$
there exist $c\ge 0$ and a finite connected cyclically reduced
$\Gamma$-graph $\Delta$ such that for every nondegenerate subtree
$K\subseteq X$ of radius $\le r$ we have $|\langle K,
\mu\rangle_\alpha-\langle K, c\mu_{\Delta}\rangle|<\epsilon$.

Choose a large integer $r\ge 1$.
Define a function $\theta: \mathcal B_{\Gamma, r}\to \mathbb R_{\ge
  0}$ by putting $\theta(K)=\langle K, \mu\rangle_\alpha$. Then
$\theta\in \mathcal Q_{\Gamma, r}$. Since the polyhedron $\mathcal Q_{\Gamma, r}$ is
defined by a finite collection of linear equations and inequalities
with rational (actually, integer) coefficients, the points with
rational coordinates are dense in $\mathcal Q_{\Gamma, r}$. Thus we
can find a nonzero $\theta'\in \mathcal Q_{\Gamma, r}$ such that for
every $K\in \mathcal B_{\Gamma, r}$ $\theta'(K)\in
\mathbb Q$ and 
$|\theta'(K)-\theta(K)|$ is arbitrarily small. 

In view of Corollary~\ref{cor:b}, if $\mu'\in \gcn$ is such that
$\theta'(K)=\langle K, \mu'\rangle_\alpha$ for every $K\in  \mathcal
B_{\Gamma, r}$ then for every finite subtree $K\subseteq X$ of radius
$\le r$ the value $|\langle \mu, K\rangle_\alpha-\langle \mu',
K\rangle_\alpha|$ is also arbitrarily small.

Choose an integer $m\ge 1$ such that for every $K\in \mathcal
B_{\Gamma, r}$ we have $m\theta'(K)\in \mathbb Z$ and put
$\theta'':=m\theta'$. By Theorem~\ref{thm:realization}, there exists a
finite cyclically reduced $\Gamma$-graph $\Delta$ such that $\langle
K, \mu_\Delta\rangle_\alpha=\theta''(K)=m\theta'(K)$ for every $K\in \mathcal
B_{\Gamma, r}$.

Let $\Delta_1,\dots, \Delta_s$ be the connected components of
$\Delta$. Put $\mu':=\frac{1}{m}\mu_\Delta=\sum_{i=1}^s
\frac{1}{m}\mu_{\Delta_i}$. Thus each $\frac{1}{m}\mu_{\Delta_i}$ is
rational and hence $\mu'$ belongs to the linear span of the set of all
rational subset currents in $\gcn$. By Proposition~5.2 of \cite{KN3},
the set $\gcnr$ of all rational currents is dense in its linear span in
$\gcn$. (Note that the proof of Proposition~5.2 in \cite{KN3} was
based on an explicit combinatorial surgery argument using large finite covers
and did not rely on the results of Bowen and Elek about unimodular graph
measures). 
Therefore there exists $\mu''\in\gcnr$ such
that $|\langle K,
\mu'\rangle_\alpha-\langle K, \mu''\rangle|$ is arbitrary small for
all finite subtrees $K\subseteq X$ of radius $\le r$. It follows that for every  finite subtree $K\subseteq X$ of radius
$\le r$ the value $|\langle \mu, K\rangle_\alpha-\langle \mu'',
K\rangle_\alpha|$ is also arbitrarily small, as required.

\end{proof}

\end{document}